\documentclass[11pt]{amsart}

\usepackage{graphicx} 
\usepackage{tikz}

\usepackage{subfig}

\usepackage{listofitems}
\usepackage{pgfplots}
\pgfplotsset{compat=1.15}
\usepackage{mathrsfs}
\usetikzlibrary{arrows}
\usepackage{enumerate}
\usepackage{mathtools}
\usepackage[scale]{mnotes}
\usepackage{dsfont}

\usepackage{soul}
\usetikzlibrary{patterns}

\usepackage[a4paper,left=37mm,right=37mm,top=35mm,bottom=40mm,marginpar=20mm]{geometry} 

\usepackage[utf8]{inputenc}
\usepackage{amsmath}
\usepackage{amsfonts}
\usepackage{amssymb}
\usepackage{amsthm}
\usepackage{esint}
\usepackage{bbm}
\usepackage{bm}
\usepackage{cancel}
\usepackage{color, comment}
\usepackage[square,sort,comma,numbers]{natbib}
\usepackage[pdftex,     
plainpages=false,   
breaklinks=true,    
colorlinks=true,
linkcolor=blue,
citecolor=blue,
pdftitle={Total variation as Gamma-limit of BMO seminorms},
pdfauthor={Adolfo Arroyo Rabasa, Paolo Bonicatto, Giacomo Del Nin}
]{hyperref} 
\usepackage[normalem]{ulem}

\usepackage{hyperref}       
\makeatletter
\def\@cite#1#2{[\textbf{#1\if@tempswa , #2\fi}]}
\makeatother

\newcommand{\Ksf}{\mathsf K}

\newcommand{\Ccal}{\mathcal{C}}
\newcommand{\Dcal}{\mathcal{D}}

\newcommand{\Hcal}{\mathcal{H}}

\newcommand{\Lcal}{\mathcal{L}}

\newcommand{\Ncal}{\mathcal{N}}
\newcommand{\Ocal}{\mathcal{O}}

\newcommand{\Qcal}{\mathcal{Q}}

\newcommand{\Zbb}{\mathbb{Z}}

\newcommand{\e}{\varepsilon}
\newcommand{\eps}{\varepsilon}
\renewcommand{\H}{\mathcal{H}}

\newcommand{\F}{\mathcal{F}}
\newcommand{\G}{\mathcal{G}}
\newcommand{\I}{\mathcal{I}}

\newcommand{\K}{{\sf K}}

\renewcommand{\P}{\mathcal{P}}
\newcommand{\Q}{\mathcal{Q}}

\newcommand{\R}{\mathbb{R}}

\newcommand*{\genbf}[1]{\ifmmode\mathbf{#1}\else\textbf{#1}\fi}

\newcommand{\Osc}{\mathrm{Osc}}
\newcommand{\loc}{\mathrm{loc}}

\newcommand{\1}{\mathbbm{1}}

\renewcommand{\1}{\mathds1}
\newcommand{\dd}{\mathrm{d}}

\newtheorem{theorem}{Theorem}
\newtheorem{lemma}[theorem]{Lemma}
\newtheorem{proposition}[theorem]{Proposition}

\newtheorem{remark}[theorem]{Remark}

\title[Total variation as \texorpdfstring{$\Gamma$}{Gamma}-limit of \texorpdfstring{$BMO$}{BMO} seminorms]{Representation of the total variation \\ as a  \texorpdfstring{$\Gamma$}{Gamma}-limit of \texorpdfstring{$BMO$}{BMO}-type seminorms}

\author{Adolfo Arroyo-Rabasa}
\address[A.\ Arroyo-Rabasa]{Université catholique de Louvain,
	Building Marc de Hemptinne,
	Chemin du Cyclotron 2,
	1348 Louvain-la-Neuve, Belgium}
\email{adolfo.arroyo@uclouvain.be}

\author{Paolo Bonicatto}
\address[P.\ Bonicatto]{Mathematics Institute, University of Warwick,
	Zeeman Building, CV4 7HP Coventry, UK}
\email{Paolo.Bonicatto@warwick.ac.uk}

\author{Giacomo Del Nin}
\address[G.\ Del Nin]{Mathematics Institute, University of Warwick,
	Zeeman Building, CV4 7HP Coventry, UK}
\email{Giacomo.Del-Nin@warwick.ac.uk}

\date{\today}

\begin{document}

	\begin{abstract}
		We address a question raised by Ambrosio, Bourgain, Brezis, and Figalli, proving that the $\Gamma$-limit, with respect to the $L^1_\loc$ topology, of a family of $BMO$-type seminorms is given by $\tfrac14$ times the total variation seminorm. Our method also
		yields an alternative proof of previously known lower bounds for the pointwise limit and conveys a compactness result in $L^1_\loc$ in terms of the boundedness of the $BMO$-type seminorms.
		\vspace{12pt}
		
		\vspace{4pt}
		
		\vspace{5pt}
		
	\end{abstract}	
	
	\maketitle
	
	\section{Introduction}
	
	In recent years there has been a significant interest in the relations between the gradient seminorm $|Df|(\R^n)$ of a $BV$ function $f : \R^n \to \R$ and the $\e$-scale $BMO$-type  seminorms
	\[
	\left\{ \fint_{Q} \Big\vert f(x) - \fint_{Q} f\Big\vert \, \dd x\right\}_{Q \in \mathcal F_\eps},
	\]
	where $\F_\e$ is a collection of disjoint $\eps$-cubes $Q \subset \R^n$.
	This question is rooted in the work of Bourgain, Brezis and Mironescu~\cite{BBM}, who introduced the space
	\begin{equation*}
		B := \{f \in L^1(Q_0):\, \|f\|_B < \infty\},
	\end{equation*}
	where $Q_0=(0,1)^n$,
	\begin{equation*}
		\|f\|_B \coloneqq \sup_{\e \in (0,1)} [f]_\eps
	\end{equation*}
	and 
	\begin{equation}\label{eq:def_[f]epsilon}
		[f]_\eps \coloneqq \sup_{\G_\eps} \eps^{n-1}  \sum_{Q \in \G_\eps} \fint_{Q} \Big\vert f(x) - \fint_{Q} f\Big\vert\, \dd x,\qquad f\in L^1(Q_0),
	\end{equation}
	where $\G_\eps$ denotes a family, with cardinality at most $\eps^{1-n}$, of disjoint $\eps$-cubes $Q \subset Q_0$ with their edges parallel to the canonical coordinate axes.
	The introduction of the space $B$ aimed at creating a common framework for various regularity results, which the following statement can broadly summarize: every function $f : Q_0 \to \mathbb Z$ that belongs to the subspace 
	\begin{equation*}
		B_0:=\left\{f \in B: \limsup_{\e \to 0}\, [f]_\e = 0 \right\}
	\end{equation*}
	is necessarily constant. One of the interesting features of $B$ is that it contains several functional spaces of interest such as $BV$, $BMO$ and $W^{1/p,p}$ for every $p \in [1,\infty)$.
	The fact that $BV \subset B$ is a direct consequence of a scaling argument and Poincar\'e's inequality for the unit cube, which convey that there exists a dimensional constant $C_n$  such that
	\begin{equation}\label{eq:poincare}
	\fint_{Q} \Big\vert f(x) - \fint_{Q} f\Big\vert \dd x \le C_n \, \frac{|Df|(Q)}{\eps^{n-1}} \qquad \text{for every $\eps$-cube $Q \subset Q_0$,}
	\end{equation}
	and therefore indeed $[f]_\eps \le C_n|Df|(Q_0)$.
	
	Notice that $f \mapsto \|f\|_B$ is anisotropic because it is not invariant under domain rotations. Motivated by this, Ambrosio, Bourgain, Brezis, and Figalli proposed a natural \emph{isotropic} modification of \eqref{eq:def_[f]epsilon} in full space. Namely, they defined the energies
	\begin{equation*}
		\mathsf I_{\e}(f) \coloneqq \sup_{\mathcal F_\e} \e^{n-1} \sum_{Q \in \mathcal F_\e} \fint_{Q} \Big\vert f(x) - \fint_{Q} f\Big\vert \, \dd x, \qquad f \in L^1_\loc(\R^n),
	\end{equation*}
	where $\mathcal F_\e$ denotes a family of cardinality at most $\e^{1-n}$, of disjoint $\eps$-cubes $Q\subset \R^n$ with \emph{arbitrary orientation}. Notice that the restriction on the cardinality of the cubes is tailored to detect only the jump part of the derivative, due to its rectifiability properties. In the search for a $BMO$-type formula for the perimeter, the authors proved that
	\begin{equation}\label{eq:perimeter}
		\lim_{\e\to 0} \mathsf I_\e(\1_E)=\frac12 \min\{1,P(E)\}, \qquad \text{$E \subset \R^n$ measurable,}
	\end{equation}
	where $P(E)$ denotes the perimeter of $E\subseteq \R^n$. This shows that, at least for characteristic functions (or more generally, $\mathbb Z$-valued functions), the isotropic energy $\mathsf I_\eps$ is well-suited to approximate the total variation norm.  
	
	To investigate general $BV$ functions, one needs to remove the cardinality constraint by considering the energy (see \cite[Sect.~4.3]{ABBF})
	\begin{equation*}
		\mathsf K_{\e}(f) \coloneqq \sup_{\mathcal H_\e} \e^{n-1} \sum_{Q \in \mathcal H_\e} \fint_{Q} \Big\vert f(x) - \fint_{Q} f\Big\vert\, \dd x, \qquad f\in L^1_\loc(\R^n),
	\end{equation*}
	where $\mathcal H_\e$ is now a family of disjoint $\eps$-cubes $Q\subset \R^n$ of {arbitrary orientation} and \emph{arbitrary cardinality}.
	These energies were studied by Ponce and Spector~\cite{PS} and by Fusco, Moscariello, and Sbordone~\cite{FMS2}. Both works discuss, among other interesting properties, the following lower and upper bound estimates for the liminf and limsup of the $\Ksf_\e$ energies (see~\cite[Prop.~5.1]{PS} when $f \in BV$, and~\cite[Prop.~2.4]{FMS2} for the general case):  
	\begin{equation}\label{eq:FMS}
		\frac 14 |Df|(\R^n)  \le \liminf_{\eps \to 0}	\mathsf K_{\e}(f) \le \limsup_{\eps \to 0}	\mathsf K_{\e}(f) \le \frac 12 |Df|(\R^n)
	\end{equation}
	for every $f\in L^1_\loc(\R^n)$. Here $|Df|(\R^n)$ denotes the extended total variation, which may attain the value $\infty$.
	It is worthwhile to remark that the lower and the upper bounds are sharp: the lower bound is attained by smooth functions, while the upper bound is attained by characteristic functions of Caccioppoli sets (cf.~\eqref{eq:perimeter}).\footnote{It can actually be proved that the last inequality in \eqref{eq:FMS} is true for every $\e>0$, without taking the limsup. This is a consequence of the fact that in \eqref{eq:poincare} one can take $C_n =\frac12$ for every $n$.
	Indeed, for the case of characteristic functions, see \cite[Appx.]{ABBF}; in the general case, it is possible to show that there is at least one (multiple of a) characteristic function among the maximizers of the functional $f \mapsto \fint_{Q_0} \vert f(x) - \fint_{Q_0} f\vert \, \dd x$ defined on the set $\{|Df|(Q_0)\leq 1\}$.}
			 
	 More generally, De Philippis, Fusco, and Pratelli proved~\cite[Cor.~6.2]{DPFP} (see also~\cite[Thm.~3.3]{FMS}) that, if $f \in SBV_\loc$, then the pointwise limit exists and equals 
	\begin{equation}\label{eq:DFP}
		\lim_{\eps \to 0}	\mathsf K_{\e}(f) = \frac14  |D^a f|(\R^n) + \frac 12 |D^j f|(\R^n).
	\end{equation}
	Here, $D^af$ and $D^jf$ denote the absolutely continuous and the jump part of $Df$. For general $BV$ functions, the  Cantor part $D^cf$ may prevent the existence of a limit (see \cite[Ex.~6.3]{DPFP}). Despite the non-existence of the pointwise limit for general $BV$ functions, Ambrosio et al.~\cite[Sect.~4.3]{ABBF} pondered  the possibility that the $\Gamma$-limit, with respect to the $L^1_{\loc}$ topology, exists nonetheless, and is precisely $\frac14 |Df|(\R^n)$, saturating the lower bound in~\eqref{eq:FMS} for \emph{all} $BV$ functions. 
	Our main result shows that this is indeed the case.
	\begin{theorem}[$\Gamma$-limit]\label{thm:main}
		The $\Gamma$-limit of the functionals $\mathsf K_\eps:L^1_\loc(\R^n)\to [0,\infty]$ with respect to the $L^1_\loc(\R^n)$ topology exists and is given by
		\[
		\Gamma\text{-}\lim \mathsf K_\eps(f) =  \frac 14 |Df|(\R^n),\qquad f\in L^1_\loc(\R^n).
		\]
	\end{theorem}
	The choice of the $L^1_\loc$ topology emanates from the following compactness result, which may be of independent interest.
	\begin{theorem}[Compactness]
		Let $\eps_j \to 0$ be an infinitesimal sequence of positive reals and let $(f_{\e_j})_{j} \subset L^1_\loc(\R^n)$ be a sequence such that 
		\[
		\liminf_{j\to\infty} {\mathsf K}_{\e_j}(f_{\e_j}) < \infty.
		\]
		Then there exist a subsequence $(\e_{j_k})_k$, a sequence of constants $(c_k)_k \subset \R$ and a function 
		\[
		f \in L^{1}_\loc(\R^n) \quad \text{with} \quad  |Df|(\R^n) < \infty,
		\]
		such that
		\[
		f_{\e_{j_k}}-c_k \to f \; \quad\text{in $L^1_{\loc}(\R^n)$ as $k \to \infty$}. 
		\]
	\end{theorem}
	\begin{remark}
		As we will prove in Proposition \ref{prop:compactness}, the limit function $f$ enjoys better global integrability, namely it belongs to $ L^{1^*}(\R^n)$. Moreover, the translation by constants cannot be relaxed due to the invariance
		\[
		\Ksf_\eps(f + c) = \Ksf_\eps(f) \qquad \text{for all $f \in L^1_\loc(\R^n)$ and $c \in \R$}.
		\]
	\end{remark}
	
	\subsection*{Comments on the proof} Next, we briefly discuss the main arguments of the proof of Theorem~\ref{thm:main}. 
	The limsup inequality is a direct consequence of the density of $C^\infty$ in $L^1_\loc$ and of the equality 
	\begin{equation}\label{eq:cheating}
		\lim_{\eps \to 0}	\mathsf K_{\e}(f) = \frac14  |D f|(\R^n) \quad \text{for all $f \in C^\infty( \R^n)$},
	\end{equation}
	which follows from \eqref{eq:DFP}.
	
	
	The proof of the liminf inequality, instead, is first discussed for functions of one variable, in which case the key argument is relatively easy to absorb and didactic towards the proof of the general case. The overly simplified geometry of the real line allows us to give a direct proof based on the construction of piecewise constant approximations $\tilde{f}_\eps$, equal to the average of $f_\eps$ over intervals induced by $\tfrac{\e}{2}\Zbb$. The crucial argument is that one can control the total variation of the approximations $\tilde{f}_\eps$ by the energy $\mathsf K_{\e}(f_\eps)$. More precisely, we prove by a direct computation (see Lemma~\ref{lemma:1/4|Dw|_leq_K_2epsilon}) the validity of the following inequality:
	\begin{equation}\label{eq:factor4}
		|D \tilde{f}_\e|(\R)\leq 4\K_{\e}(f_\e).
	\end{equation}
	Since the approximations $\tilde{f}_\eps$ still converge to $f$ in $L^1_\loc(\R)$, the lower semicontinuity of the total variation yields the sought lower bound
	\begin{equation}\label{eq:lower_bound_intro}
	\frac{1}{4}|Df|(\R) \le  \liminf_{\eps \to 0} \K_{\e}(f_\e).
	\end{equation}
	
	We note here that our proof conveys, in a rather transparent and clear fashion, the geometric meaning of the prefactor $\frac 14$ (or, equivalently, of the factor $4$ in \eqref{eq:factor4}). It arises from two estimates with factor $2$, which account for two independent aspects. 
	The first factor $2$ comes from a doubling of scale and appears when estimating the difference between the averages of $f_\e$ on two adjacent intervals $I,I'$ of size $\tfrac{\e}{2}$, in terms of the oscillation on $I\cup I'$:
	\begin{equation}\label{eq:diff_average_intro_bis} 
	\Big\vert  \fint_{I} f_\e -  \fint_{I'} f_\e \Big \vert \le 2 \fint_{I \cup I'} \Big\vert f_\e(x) - \fint_{I \cup I'}  f_\e\Big\vert\, \dd x. 
	\end{equation}
	We now notice that, by construction, $D\tilde f_\e$ is purely atomic (it is concentrated on $\frac \eps2 \mathbb Z$) and its total variation can be estimated from above by summing up the right-hand side of \eqref{eq:diff_average_intro_bis} over all intervals $I$ induced by $\frac{\e}{2} \Zbb$.
	This is precisely where the second factor $2$ comes from. When performing the said sum, the intervals $I\cup I'$ have non-trivial overlap, and split naturally into \emph{two} partitions of $\R$ in $\e$-intervals. Since $\Ksf_\e(f_\e)$ is the supremum among all such partitions, we lose here another factor $2$.

	In dimension $n>1$, the same strategy yields
	\begin{equation*}
		|D \tilde{f}_\e|(\R^n)\leq 4n\K_{2\e}(f_\e),
	\end{equation*}
	which is, however, non-optimal.
	To obtain the optimal constant in general dimension, we thus make use of the Fonseca--M\"uller \emph{blow-up strategy}~\cite{FM} (see also \cite{S}, which was the initial inspiration for our approach) and we exploit the infinitesimal \emph{one-directional} rigidity of real-valued $BV$ functions. The heuristic idea can be explained as follows: let us write 
	\[
	\mathbb S^{n-1} \ni  \nu(x) = \frac{\dd Df}{\dd |Df|}(x) \coloneqq \lim_{r \to 0} \frac{Df(B_r(x))}{|Df|(B_r(x))}
	\]
	to denote the polar of $Df$, which is defined $|Df|$-almost everywhere. Then, given a Lebesgue point $x_0$ of the polar, and denoting by $\nu_0\coloneqq \nu(x_0)$ its value, the blow-up method allows us to approximate, at small scales around $x_0$, the measure $|Df|$ with $D_{\nu_0}f\coloneqq \nu_0\cdot Df$. This is the crucial part of this strategy, since by working with the one-directional derivative $D_{\nu_0} f$, we can reduce the problem to the case $n = 1$. 

	\subsection*{Concluding remarks.} Prior to~\cite{BBM}, Bourgain, Brezis, and Mironescu~\cite{BBM2} and D\'avila~\cite{Davila} considered the approximation of the total variation seminorm by \emph{convolution}\emph{-type} functionals $\mathscr F_\e$.
	In principle, such convolution energies and the functionals $\Ksf_\e$ are comparable up to a dimensional constant (see, e.g.,~\cite[Prop.~5.2]{PS}), but the fine convergence properties in these two settings differ significantly. For instance, the pointwise limit $\lim_{\e \to 0} \mathscr F_\e(u)$ exists for all $u \in L^1$, while $\lim_{\e \to 0} \Ksf_\e(u)$ might fail to exist for $u \in L^1$ (cf.~\cite[Ex.~6.3]{DPFP}). Very general $\Gamma$-convergence results have already been established for convolution-type functionals (see, e.g., \cite{P, N}). However, to the best of our knowledge, our result is the first that is concerned with a \emph{discrete-type} energy similar to the one considered in \cite{BBM}.
	
	\subsection*{Addendum}
	After the first pre-print of this work appeared, Ha\"im Brezis shared with us an unpublished work by Nguyen~\cite{Hung}, which contains an alternative approach to the lower bound of the $\Gamma$-limit. Nguyen's argument is based on a simple convolution estimate of the form
	\[
	\Ksf_\e(f \ast \rho_\delta) \le \Ksf_\e(f), \qquad f \in L^1_\loc(\R^n)
	\]
	where $(\rho_\delta)_\delta$ is a family of convolution kernels, and on the validity of the identity $\frac 14 |D(f \ast \rho_\delta)| = \lim_{\e \to 0} \Ksf_\e(f \ast \rho_\delta)$, which follows from~\eqref{eq:DFP} and from the smoothness of $f \ast \rho_\delta$ for $\delta>0$. On the other hand, our approach, based on piecewise constant approximations rather than convolutions, serves as a building block for both the lower bound and the compactness theorem. It also gives a different account of the appearance of the prefactor $\frac14$, which was only known to be related to the optimal Poincar\'e constant for linear maps on the unit cube (cf. \cite[Lemma 3.1]{FMS}). In particular, we remark that our proof yields an alternative explanation of the pointwise lower bound 
	\[
		\frac 14 |Df| \le \liminf_{\e \to 0} \Ksf_\e(f),
	\]
	by considering the constant family $f_\e = f$ in the $\Gamma$-liminf inequality.
		
	\subsection*{Acknowledgements} A.A.-R was supported by the Fonds de la Recherche Scientifique - FNRS under Grant No 40005112 (Compensated compactness and fine properties of PDE-constrained measures). P.B. and G.D.N. have received funding from the European Research Council (ERC) under the European Union's Horizon 2020 research and innovation programme, grant agreement No 757254 (Singularity).

	\section{Notation and preliminaries}
	
	We assume that $n \ge 1$ is an integer throughout the paper. If $A$ is a set, we denote by $A^c$ its complement and by $\1_A$ its characteristic function. We write $B_r(x)$ to denote the open ball of radius $r>0$ centered at a point $x \in \R^n$, and $D_r$ to denote the cube $(-r,r)^n$.
	To denote the $n$-dimensional Lebesgue (outer) measure of a set $A\subset \R^n$, we write $\Lcal^n$  or, when no risk of confusion arises, simply $|A|$. As usual, we write $\Hcal^{n-1}$ to denote the Hausdorff $(n-1)$-dimensional measure. 
	The standard Lebesgue spaces are defined in the usual way, and they are denoted by $L^p(\R^n)$, omitting the domain when there is no risk of confusion. The same holds for their local counterpart, i.e., we write $f \in L^p_{\loc}(\R^n)$ if $f \1_K \in L^p(\R^n)$ for every compact $K \subset \R^n$.
	The same carries over to the Lebesgue spaces associated with any non-negative Borel measure $\mu$, which will be written as $L^p(\R^n,\mu)$. 
	Finally, we write $\mathbb S^{n-1}$ to denote the unit sphere in $\R^n$. 
	
	\subsection{\texorpdfstring{$BV$}{BV} functions}
	Given a function $f\in L^1_\loc(\R^n)$, we denote by $|Df|(\R^n)$ its extended total variation defined by
	\[
	|Df|(\R^n)\coloneqq \sup\left\{\int_{\R^n} f(x)\,\mathrm{div}\phi(x)\,\dd x:\, \phi\in C^1_c(\R^n,\R^n), \,\sup_{x \in \R^n} |\phi(x)| \leq 1\right\},
	\]
	which may attain the value $\infty$. By construction $f \mapsto |Df|(\R^n)$ is lower semicontinuous with respect to the $L^1_\loc$ topology. By Riesz's theorem, if $|Df|(\R^n)<\infty$ then the distributional derivative $Df = (D_1f, \ldots , D_nf)$ of $f$ is a vector-valued measure with finite total variation. We say that $f$ has bounded variation, in symbols $f \in BV$, if $f \in L^1$ and $|Df|(\R^n)<\infty$. Appealing to the Radon-Nikodym's theorem, we shall often write $Df = \nu |Df|$ to denote the polar decomposition of $Df$, where 
	\begin{equation}\label{eq:def_polar}
		\nu(x) := \frac{\dd Df}{\dd |Df|}(x) =\lim_{r\to 0}\frac{Df(B_r(x))}{|Df|(B_r(x))}
	\end{equation}
	is well-defined and of unit length at $|Df|$-a.e. $x \in \R^n$.
	Given a direction $e \in \mathbb S^{n-1}$, we define the partial directional derivative in direction $e$ of $f$ as the measure $D_e f := (\nu \cdot e) |Df|$.
	
	Let $f \in L^1_\loc(\R^n)$ be such that $|Df|(\R^n)< \infty$. We introduce the set $S \subset \R^n$ made up of points $x_0 \in \R^n$ satisfying the following two conditions:  
	\begin{enumerate}[(a)]
		\item the limit in \eqref{eq:def_polar} exists at $x_0$ and $\nu_0 := \nu(x_0)$ has unit norm, i.e. $|\nu_0|=1$; 
		\item $x_0$ is a Lebesgue point of the map $\nu \in L^\infty(\R^n,|Df|)$, that is 
		\[
		\lim_{r\to 0}\fint_{B_r(x_0)} |\nu(x_0)-\nu(x)|\,\dd|Df|(x)=0.
		\]
	\end{enumerate}
	Radon-Nikodym's theorem and Lebesgue's differentiation theorem (see \cite[Cor. 2.33]{AFP}) convey that $S$ has full $|Df|$-measure. 
	
	For future use, we record here a well-known \emph{mono-directionality property} satisfied by functions with bounded variation. More precisely, the following Lemma shows that, infinitesimally, at every point of $S$ (in particular, at $|Df|$-a.e. point), the function $f$ oscillates with respect to only one direction. 
	
	\begin{lemma}[Mono-directionality of $BV$ functions]\label{lemma:monodirectional} 
		For $|Df|$-a.e. $x_0\in\R^n$, it holds 
		\begin{equation}\label{eq:directional_derivative}
			\lim_{r\to 0}\frac{D_{\nu_0}f (B_r(x_0))}{|Df|(B_r(x_0))}=1. 
		\end{equation}
	\end{lemma}
	
	\begin{proof} We show that \eqref{eq:directional_derivative} is satisfied at every point $x_0$ that belongs to the set $S$ defined above. Indeed,
		\begin{align*}
			\fint_{B_r(x_0)} |\nu_0-\nu(x)|\,\dd|Df|(x)&\geq \bigg|\fint_{B_r(x_0)} (\nu_0-\nu(x))\,\dd|Df|(x)\bigg|\\
			&=\frac{1}{|Df|(B_r(x_0))}\Big|\nu_0 |Df|(B_r(x_0))-Df(B_r(x_0))\Big|\\
			&\geq \frac{1}{|Df|(B_r(x_0))}\Big||Df|(B_r(x_0))-\nu_0\cdot Df(B_r(x_0))\Big|\\
			&=\left|1-\frac{D_{\nu_0}f(B_r(x_0))}{|Df|(B_r(x_0))}\right|,
		\end{align*}
		and, since $x_0 \in S$, the left-hand side goes to zero as $r\to 0$.
	\end{proof} 
	
	\begin{remark}[Lebesgue's differentiation theorem with cubes]\label{rem:general_lebesgue} 
		In the following, we will need to work with a family of possibly rotated cubes $Q_r(x_0)$ centered at $x_0$ and of side-length $2r$. Since Lebesgue's differentiation theorem holds (for general Borel measures) for more general families of sets (see~{\cite[Thm. 3.2]{F}}), there is no loss of generality in assuming that \eqref{eq:directional_derivative} holds at $|Df|$-a.e. $x_0\in \R^n$, if we replace $B_r(x_0)$ with $Q_r(x_0)$.
	\end{remark}

	\subsection{On \texorpdfstring{$\Gamma$}{Gamma}-convergence and structure of the paper}
	According to the definition of $\Gamma$-convergence (see, e.g., the monograph \cite{DM}), to prove Theorem~\ref{thm:main} we will need to  verify the following $\Gamma$-limsup and $\Gamma$-liminf inequalities: 
	
	\begin{enumerate}[(i)]
		\item {\bf $\Gamma$-liminf}: for every family $(f_\eps)_{\eps>0} \subset L_\loc^1(\R^n)$ such that $f_\e \to f$ in $L^1_\loc$ as $\eps \to 0$ it holds 
		\begin{equation*}
			\frac{1}{4} |Df|(\R^n) \le \liminf_{\eps \to 0} \mathsf K_{\e} (f_\e).
		\end{equation*}
		\item {\bf $\Gamma$-limsup}: for every $f \in L_\loc^1(\R^n)$, there exists a family $(f_\eps)_{\eps>0} \subset L_\loc^1(\R^n)$ such that $f_\e \to f$ in $L^1_\loc$ as $\eps \to 0$ and 
		\begin{equation*}
			\limsup_{\eps \to 0} \mathsf K_{\e} (f_\e) \le \frac{1}{4} |Df|(\R^n). 
		\end{equation*}
	\end{enumerate}
	The $\Gamma$-liminf inequality is rather delicate, and its proof is split across multiple sections of the paper. In Section~\ref{sec:lower_bound_one} we address the proof of the main result for functions of the real line ($n=1$), and in Section~\ref{sec:lower_bound_general_case} we tackle the general case for functions of several variables by refining the proof of the main compactness result (which will be established in Section~\ref{sec:compactness}). Finally, Section~\ref{sec:upper_bound} contains the simple proof of the $\Gamma$-limsup inequality. 
	
	\section{The lower bound in dimension one}\label{sec:lower_bound_one}
	
	In this section, we give an elementary argument for the lower bound in dimension one ($n = 1$).
	\begin{proposition}\label{prop:lower_bound_dimension_1}
		Let $(f_\eps)_{\eps>0} \subset L_\loc^1(\R)$ be such that $f_\e \to f$ in $L^1_\loc(\R)$ as $\eps \to 0$. Then 
		\[
		\frac14 |Df|(\R)\leq \liminf_{\e\to 0}\mathsf K_{\e}(f_\e).
		\]
	\end{proposition}
	
	Before embarking on the details of the proof itself, let us briefly comment on the main strategy. The heart of the argument will be to replace the family $(f_\e)_\e$ with another family $(\tilde f_\e)_\e$ that is piecewise constant on intervals of length $\e/2$ (equal to the average of $f_\e$ over that interval), which still converges to $f$ in $L^1_\loc$. Since $\tilde f_\e$ is piecewise constant, its derivative is purely atomic, and we exploit this simple structure to prove that
	\begin{equation}\label{eq:step2}
		|D\tilde f_\e|(\R) \le 4\Ksf_{\e}(f_\e).
	\end{equation} 
	The conclusion then follows from the lower semicontinuity of the total variation with respect to the convergence $\tilde f_\e \to f$ in $L^1_\loc$.
	
	To make this more rigorous, let us introduce the following notation: for a given $\delta>0$, we write 
	\[
	\I^\delta:=\left\{I^\delta_j : j \in \mathbb Z \right\},
	\]
	to denote the family of disjoint open sub-intervals 
	\[
	I^\delta_j:=((j - 1)\delta,j\delta), 
	\]
	induced by the lattice $\delta \mathbb Z$. For any $w \in L_\loc^1(\R)$, we may define a $\delta$-scale piecewise approximation by setting
	\begin{equation}\label{eq:def_w_delta}
		w^\delta \coloneqq \sum_{I \in \mathcal I^\delta} \1_{I} w_{I},  \qquad w_{I} \coloneqq  \fint_{I} w(x)\,\dd x.
	\end{equation}
	Roughly speaking, the linear operator $w \mapsto w^\delta$ could be considered a discrete mollification. It features some of the usual properties of mollification, as it does not increase (locally) the $L^1$-norm.
	We collect these properties in the following lemma:
	
	\begin{lemma}\label{lemma:1/4|Dw|_leq_K_2epsilon}
		Let $w \in L_\loc^1(\R)$ and let $\delta > 0$. Then:
		\begin{enumerate}
			\item[\rm (i)] for each $\delta >0$ and every bounded interval $[a,b] \subset \R$, it holds
			\[
			\Vert w^\delta \Vert_{L^1([a,b])}\le \|w\|_{L^1([a-\delta,b+\delta])};
			\] 
			\item[\rm (ii)] $w^\delta \to w$ in $L^1_\loc(\R)$ as $\delta \to 0$;
			\item[\rm (iii)] for each $\delta >0$, it holds
			\[
			\frac 14 |D w^\delta|(\R) \le \Ksf_{2\delta}(w). 
			\]
		\end{enumerate}
	\end{lemma}
	
	\begin{proof}
		{\rm (i)} For every interval $I\in\I_{\delta}$ it  holds
		\[
		\int_I | w^{\delta}(x)|\, \dd x\leq \int_I \left(\frac{1}{|I|}\int_I |w(y)|\, \dd y\right) \dd x=\|w\|_{L^1(I)}.
		\]
		Now we consider the smallest set $U$ that is made up of unions of intervals in $\mathcal{I}_\delta$ and that contains $[a,b]$ up to an $\mathcal{L}^1$-negligible set. With this consideration in mind, we deduce that
		\begin{align*}
			\Vert w^\delta \Vert_{L^1([a,b])}& \leq\Vert w^\delta \Vert_{L^1(U)} \\
			& =\sum_{\substack{I\in \mathcal{I}_\delta\\I\subseteq U}} \int_I| w^{\delta}(x)|\, \dd x\leq \|w\|_{L^1(U)}\leq \|w\|_{L^1([a-\delta,b+\delta])}.
		\end{align*}
		
		Next, we prove {\rm (ii)}. Let us fix a bounded interval $[a,b]$. First we assume that $w$ is in $L^\infty ([a-\delta,b+\delta])$. By the uncentered version of Lebesgue's differentiation theorem (\cite[Thm.~3.2]{F}), $w^\delta\to w$ pointwise a.e. in $[a,b]$, and by the assumption $w^\delta$ is uniformly bounded in $L^\infty([a,b])$. By dominated convergence $w^\delta\to w$ in $L^1([a,b])$.
		To address the general case we argue as follows: for any given $\eta>0$, we can find $v\in L^\infty([a-1,b+1])$ such that $\|v-w\|_{L^1([a-1,b+1])}<\eta$. Then, the previous step gives
		\[
		\|v-v^\delta\|_{L^1([a,b])}\to 0\quad\text{as $\delta\to 0$}.
		\]
		Moreover, by point {\rm (i)} applied to $w-v$, we have that 
		\[
		\|w^\delta-v^\delta\|_{L^1([a,b])}\leq \|w-v\|_{L^1([a-\delta,b+\delta])}<\eta.
		\]
		Hence, by the triangle inequality we deduce the approximation bound
		\begin{align*}
			\limsup_{\delta\to 0}\|w-w^\delta\|_{L^1([a,b])} & \leq \limsup_{\delta\to 0} \big(\|w-v\|_{L^1([a,b])} \\
			& \qquad +\|v-v^\delta\|_{L^1([a,b])}+\|v^\delta-w^\delta\|_{L^1([a,b])}\big)\\
			& <2\eta.
		\end{align*}
		Since this holds for every $\eta>0$ and every bounded interval $[a,b]$, we have shown that $w^\delta\to w$ in $L^1_\loc$.
		
		It remains to show {\rm (iii)}. For any two neighboring intervals $I,I'\in \I^\delta$ and $m \in \R$, we may estimate the difference of the averages $|w_I-w_{I'}|$ as
		\[
		\left|\fint_I w(x)\,\dd x-\fint_{I'}w(x)\,\dd x\right| \leq\fint_I|w(x)-m|\,\dd x+\fint_{I'}|w(x)-m|\,\dd x.
		\]
		In particular, choosing $m\coloneqq\fint_{I\cup I'}w$, we obtain the estimate
		\[
		|w_I-w_{I'}|\leq 2 \fint_{I\cup I'} \left|w(x)-\fint_{I\cup I'}w\right|\,\dd x.
		\]
		Since $Dw^\delta$ is purely atomic, we further get
		\begin{align*}
			|Dw^\delta|(\R)&=\sum_{j \in \mathbb Z}  |w_{I^\delta_j}-w_{I^\delta_{j+1}}| \\
			& \leq \sum_{j \in \mathbb Z}   2\fint_{I_j^\delta\cup I^\delta_{j+1}}\Big|w(x)-\fint_{I^\delta_j\cup I^\delta_{j+1}}w\Big|\, \dd x\\
			& =2 \bigg(\sum_{I\in \P_1}\fint_I \Big|w(x)-\fint_I w\Big|\, \dd x+\sum_{I\in \P_2}\fint_I \Big|w(x)-\fint_I w\Big|\, \dd x\bigg),
		\end{align*}
		where we denoted by $\P_1$ and $\P_2$ the collection of intervals of length $2\delta$, obtained by considering only $I^\delta_j\cup I^\delta_{j+1}$ for  $j$ even and $j$ odd, respectively. From this computation we conclude that
		\begin{equation*}
			|D w^\delta|(\R)\leq 4 \sup_{\H_{2\delta}}\sum_{I\in\H_{2\delta}}\fint_I\Big| w(x)-\fint_I w\Big|\, \dd x=4\K_{2\delta}(w),
		\end{equation*}
		as desired.
	\end{proof}
	
	We can now address the proof of the lower bound inequality for functions on the real line.
	\begin{proof}[Proof of Proposition~\ref{prop:lower_bound_dimension_1}] 
		Recalling the notation introduced in \eqref{eq:def_w_delta}, we set for $\e>0$, 
		\[
		\tilde{f}_\e := f_\e^{\e/2}
		\]
		and we prove that
		\begin{equation}\label{eq:stable}
			\tilde{f}_\e \to f \; \text{in $L^1_\loc(\R)$}.
		\end{equation}
		By triangle inequality and by Lemma \ref{lemma:1/4|Dw|_leq_K_2epsilon} {\rm (i)}, we have for every bounded interval $[a,b] \subset \R$
		\begin{align*}
			\Vert \tilde{f}_\e - f \Vert_{L^1([a,b])} & \le \Vert f^{\e/2}_\e - f^{\e/2} \Vert_{L^1([a,b])} + \Vert  f^{\e/2} -f \Vert_{L^1([a,b])} \\
			& \le \Vert f_\e - f \Vert_{L^1([a-1,b+1])} + \Vert  f^{\e/2} -f \Vert_{L^1([a,b])}
		\end{align*}
		for $\e>0$ sufficiently small. Since $f_\e \to f$ in $L^1_\loc(\R)$ the first term in the right-hand side goes to $0$. For the second term, Lemma \ref{lemma:1/4|Dw|_leq_K_2epsilon} {\rm (ii)}, yields $f^{\e/2} \to f$ in $L^1_\loc(\R)$ as $\e \to 0$ and we thus conclude \eqref{eq:stable}. 
		
		Having established this convergence, we can now prove the $\Gamma$-liminf inequality. Indeed, Lemma \ref{lemma:1/4|Dw|_leq_K_2epsilon} {\rm (iii)} with $w := f_\e$ and $\delta :=\frac{\e}{2}$ gives 
		\[
		\frac 14 |D \tilde{f}_\eps|(\R)  \le \Ksf_{\e}(f_\e).
		\]
		Combining this with \eqref{eq:stable} (and the lower semicontinuity of the total variation with respect to $L^1_\loc$-convergence), we obtain
		\[
		\frac14|Df|(\R)\leq \liminf_{\e \to 0}\frac14|D\tilde{f}_\e|(\R)\leq \liminf_{\e \to 0} \mathsf K_{\e}(f_\e),
		\]
		which is the sought inequality.
	\end{proof}

	\section{Compactness in general dimension}\label{sec:compactness}
	We now consider the higher dimensional case and study the problem when $n>1$. We introduce a construction similar to the one of the previous section by defining an approximation that is piecewise constant on \emph{cubes} of side $\e$. This, however, \emph{will not} give the sharp inequality for the lower bound of the $\Gamma$-limit. Nonetheless, it will be enough to obtain the following compactness result:
	
	\begin{proposition}[Compactness]\label{prop:compactness} 
		Let $\eps_j \to 0$ be an infinitesimal sequence of positive reals and let $(f_{\e_j})_{j} \subset L^1_\loc(\R^n)$ be a sequence such that 
		\[
		\liminf_{j\to\infty} {\mathsf K}_{\e_j}(f_{\e_j}) < \infty.
		\]
		Then there exist a subsequence $(\eps_{j_k})_k$, a sequence of constants $(c_k)_k \subset \R$ and a function 
		\[
		f \in L^{1}_\loc(\R^n) \quad \text{with} \quad  |Df|(\R^n) < \infty,
		\]
		such that
		\begin{equation}\label{eq:convergence_compactness}
			f_{\e_{j_k}}-c_k \to f \; \quad\text{in $L^1_{\loc}(\R^n)$ as $k \to \infty$}. 
		\end{equation}
		In addition, $f\in L^{1^*}(\R^n)$ and
		\begin{equation}\label{eq:tot_var_in_compactness}
			\frac{1}{4n}|Df|(\R^n)\leq \liminf_{j\to\infty} \mathsf{K}_{\e_{j}}(f_{\e_{j}}).
		\end{equation}
	\end{proposition}
	As in the one-dimensional case, before passing to the proof of this result, we need to introduce the piecewise $\eps$-approximations of $f$, as well as several elementary relations between $\Ksf_{\eps}(f)$ and both the $L^1$ norm and the total variation of such approximations.
	
	\subsection{Properties of the piecewise \texorpdfstring{$\delta$}{delta}-approximations}\label{sec:properties}
	We now consider the mesh of disjoint open cubes of side $\delta>0$ and edges parallel to the coordinate axes (called $\delta$-cubes):
	\begin{equation}\label{eq:def_mesh_cubes}
		\Q_\delta:=\{(0,\delta)^n+\delta z : z \in \Zbb^n\}.
	\end{equation}
	For any $\tau\in \R^n$ we will write
	\begin{equation}\label{eq:def_mesh_cubes_translated}
		\tau+\Q_\delta\coloneqq \{\tau+Q:Q\in\Q_\delta\}
	\end{equation}
	to denote the mesh $\Q_\delta$ translated by $\tau$.
	
	As before, given a function $w \in L^1_\loc(\R^n)$, we define the piecewise constant approximation at scale $\delta$ as  
	\begin{equation}\label{eq:def_w_delta_in_every_dimension}
		w^\delta(x) \coloneqq \sum_{Q \in \Qcal_{\delta}} \1_Q w_Q, \qquad w_Q \coloneqq \fint_Q w(x) \, \dd x. 
	\end{equation} 
	Just as in the one dimensional case, the estimate 
	\begin{equation}\label{eq:difference_averages}
		|w_Q-w_{Q'}| \leq \fint_Q |w(x)-m|\, \dd x+\fint_{Q'} |w(x)-m|\, \dd x
	\end{equation}
	remains valid for any pair of cubes $Q,Q' \in \Qcal_\delta$ and every $m \in \R$. This, as before, will allow us to estimate $|Dw^\delta|(\R^n)$ in terms of $\mathsf K_{2\delta}(w)$, except that in this case the estimate will carry a geometric constant that is smaller than the prefactor of the lower bound. 
	The following lemma is the analog of Lemma \ref{lemma:1/4|Dw|_leq_K_2epsilon} in the higher dimensional framework. 
	
	\begin{lemma}\label{lemma:properties_of_piecewise_general_dimension}
		Let $w \in L_\loc^1(\R^n)$ and let $\delta > 0$. Then:
		\begin{enumerate}
			\item[\rm (i)] for each $\delta >0$ and every cube $D_R= (-R,R)^n$, it holds
			\[
			\Vert w^\delta \Vert_{L^1(D_R)}\le \|w\|_{L^1(D_{R+\delta})};
			\]
			\item[\rm (ii)] $w^\delta \to w$ in $L^1_\loc(\R^n)$ as $\delta \to 0$; 
			\item[\rm (iii)] for each $\delta >0$, it holds
			\[
			\frac{1}{4n} |D w^\delta|(\R^n) \le \Ksf_{2\delta}(w).
			\]
		\end{enumerate}
	\end{lemma}

	\begin{proof}

		The proofs of Points {\rm (i)} and {\rm(ii)} are completely analogous to the ones given in Lemma \ref{lemma:1/4|Dw|_leq_K_2epsilon} in the one dimensional case and they are therefore omitted.  
		
		We address {\rm (iii)}. Let us fix $P\in\Q_{2\delta}$ and let us denote by $\Ncal_\delta(P)$ the family of unordered pairs $\{Q,Q'\}$ of distinct cubes $Q,Q'\in \Q_\delta$ that lie inside $P$ and are neighbours, i.e. they share an $(n-1)$-dimensional face, see Fig.~\ref{fig:cubes_neighbour_second}.
		\begin{figure}
			\captionsetup{width=.45\columnwidth}
			\subfloat[The mesh $\mathcal Q_{\delta}$ in $\mathbb R^2$ and some of its cubes $Q_1,Q_2,Q_3,Q_4$ contained in a cube $P \in \mathcal Q_{2\delta}$ (black, thick). The $\delta$-cube $Q_4$ has $Q_2$ and $Q_3$ as neighbours. Thus $\{Q_4,Q_3\}$, $\{Q_4, Q_2\} \in \mathcal N_\delta(P)$.]
			{
				\begin{tikzpicture}[scale=.8]
					\filldraw[white] (0,-4) circle (2.5pt);
					
					\draw[<->] (0,-.3) -- (1,-.3);
					\filldraw (.5,-.5) circle (.0pt) node {\scriptsize  $\delta$};
					
					\draw[<->] (-.3,0) -- (-.3,1);
					\filldraw (-.5,.5) circle (.0pt) node {\scriptsize  $\delta$};
					
					
					\draw[fill=gray!20] (0, 0) -- (0,1) -- (2,1) -- (2,0) -- (0,0);
					\draw[fill=gray!20] (0, 1) -- (0,2) -- (2,2) -- (2,1) -- (0,1);
					
					\filldraw (.5,1.5) circle (.0pt) node {\scriptsize  $Q_1$};
					\filldraw (1.5,1.5) circle (.0pt) node {\scriptsize  $Q_2$};
					\filldraw (.5,.5) circle (.0pt) node {\scriptsize  $Q_3$};
					\filldraw (1.5,.5) circle (.0pt) node {\scriptsize  $Q_4$};
					
					\foreach \x in {-3,...,3}{
						\draw[gray!50] (\x,-3.5) -- (\x,3.5);
					}
					\foreach \y in {-3,...,3}{
						\draw[gray!50] (-3.5,\y) -- (3.5,\y);
					}
					
					\draw[color=black, very thick] (0,0) -- (2,0) -- (2,2) -- (0,2) -- (0,0);
					
					\draw[->] (-0.5,2.5) -- (-.1,2.1);
					\filldraw (-0.7,2.7) circle (.0pt) node {\scriptsize  { $P$}};
					
					\filldraw (0,0) circle (2.5pt);
					
				\end{tikzpicture}
				\label{fig:cubes_neighbour_second}
			}\qquad
			\subfloat[The mesh $\Qcal_{2\delta}$ (dashed) in $\R^2$ and the translated mesh $(\delta,\delta)+ \Qcal_{2\delta}$ (gray). Every jump of $w^\delta$ happens across a 1-dimensional face of a cube in $\Q_\delta$. These faces appear exactly once in the \emph{interior} of a cube $P$ belonging to either $\Q_{2\delta}$ or $(\delta,\delta)+\Q_{2\delta}$. ]
			{
				\begin{tikzpicture}[scale=.8]
					
					\filldraw[white] (0,-4) circle (2.5pt);
				
					\draw[->,  very thick ] (.1,.1) -- (.95,.95);
					
					\foreach \x in {-3,-1,...,3}{
						\draw[black!10, line width=2pt] (\x,-3.5) -- (\x,3.5);
					}
					\foreach \y in {-3,-1,...,3}{
						\draw[black!10, line width=2pt] (-3.5,\y) -- (3.5,\y);
					}
					
					\foreach \x in {-2,0,2}{
						\draw[black, loosely dashed, thick] (\x,-3.5) -- (\x,3.5);
					}
					\foreach \y in {-2,0,2}{
						\draw[black, loosely dashed, thick] (-3.5,\y) -- (3.5,\y);
					}
					
					\filldraw (0,0) circle (2.5pt);	
				\end{tikzpicture}
				\label{fig:translation}
			}
			\caption{}
			\label{fig:cubes_neighbour}
		\end{figure}
		Given $\{Q,Q'\}\in \Ncal_\delta(P)$, we may estimate 
		\[
		|w_Q-w_{Q'}|=\left|\fint_Q w(x)\,\dd x-\fint_{Q'}w(x)\,\dd x\right| \leq\fint_Q|w(x)-m|\,\dd x+\fint_{Q'}|w(x)-m|\,\dd x
		\]
		for any $m \in \R$. In particular, choosing $m\coloneqq\fint_{P}w$, we obtain the estimate
		\begin{equation}\label{eq:jump_betwen_Q_Q'}
			|w_Q-w_{Q'}|\leq \frac{1}{\delta^{n}}\int_{Q \cup Q'} \Big|w(x)-\fint_{P}w\Big|\,\dd x.
		\end{equation} 
		
		We now want to estimate the derivative of $w^\delta$ in $P$, taking advantage of the fact that it is purely jump-type. More precisely, by the very definition of $w^\delta$ and exploiting the inequality \eqref{eq:jump_betwen_Q_Q'}, we have 
		\begin{align*}
			|D w^\delta|(P)&\stackrel{\phantom{\eqref{eq:jump_betwen_Q_Q'}}}=\delta^{n-1}\sum_{\{Q,Q'\}\in \Ncal_\delta(P)}|w_Q-w_{Q'}|\\
			& \stackrel{\eqref{eq:jump_betwen_Q_Q'}}\leq \delta^{n-1} \sum_{\{Q,Q'\}\in \Ncal_\delta(P)}\frac{1}{\delta^{n}} \int_{Q\cup Q'}\Big|w(x)-\fint_P w\Big|\, \dd x\\
			&\stackrel{\phantom{\eqref{eq:jump_betwen_Q_Q'}}}\leq \delta^{n-1} n\frac{1}{\delta^{n}} \int_{P}\Big|w(x)-\fint_P w\Big|\, \dd x
		\end{align*}
		where we have used that every $\delta$-cube appears in exactly $n$ unordered pairs in $\Ncal_\delta(P)$, i.e., it holds
		\[
		n \1_P(x) = \sum_{\{Q,Q'\}\in \Ncal_\delta(P)} \1_{Q\cup Q'}(x) \quad\text{for a.e. $x \in \R^n$}.
		\]
		Using now that $\Lcal^n(P)=(2\delta)^n$ we get 
		\begin{equation*}
			|D w^\delta|(P) \leq 2n (2\delta)^{n-1} \fint_{P}\Big|w(x)-\fint_P w\Big|\, \dd x. 
		\end{equation*}
		Summing up among all $P\in\Q_{2\delta}$ yields the following estimate:
		\begin{equation}\label{eq:estimate_der_P}
			\sum_{P\in\Q_{2\delta}} |Dw^\delta|(P)\leq 2n (2\delta)^{n-1}\sum_{P\in\Q_{2\delta}} \fint_P\Big| w(x)-\fint_P w\Big|\, \dd x \leq 2n  \mathsf K_{2\delta}(w).
		\end{equation} 
		It now remains to observe that 
		\begin{equation}\label{eq:u_epsilon_nD}
			|Dw^\delta|(\R^n)=\sum_{P\in \Q_{2\delta}} |Dw^\delta|(P)+\sum_{P\in (\delta,\ldots,\delta)+\Q_{2\delta}} |Dw^\delta|(P),
		\end{equation}
		which follows from the fact that every $(n-1)$-dimensional face of a cube in $\Q_\delta$ appears exactly once in the interior of a cube $P$ belonging to either $\Q_{2\delta}$ or $(\delta,\ldots,\delta)+\Q_{2\delta}$, see Fig. \ref{fig:translation}. Thus, combining \eqref{eq:u_epsilon_nD} with \eqref{eq:estimate_der_P} (and the analogous estimate for the translated partition $(\delta,\ldots,\delta)+\Q_{2\delta}$), we obtain 
		\begin{equation*}
			|Dw^\delta|(\R^n) \leq 2(2n)  \mathsf K_{2\delta}(w) = 4n  \mathsf K_{2\delta}(w), 
		\end{equation*}
		which concludes the proof. 
	\end{proof}
		
	The following quantitative estimate, which follows directly from the definition of $\Ksf_\delta$, will be helpful to prove our compactness result. 
	\begin{lemma}\label{lemma:e_close}
		Let $w \in L_\loc^1(\R^n)$. Then, for every $\delta>0$, it holds 
		\[
		\|w - w^{\delta}\|_{L^1(\R^n)} \le \delta\Ksf_{\delta}(w).
		\]
	\end{lemma}
	
	\begin{proof}
		Consider a partition $\Qcal_\delta$ of $\R^n$ (up to $\Lcal^n$-negligible sets) into open cubes of side-length $\delta$. Then, the desired estimate is a consequence of the following elementary computation:
		\begin{align*}
			\Vert w - w^\delta \Vert_{L^1(\R^n)} & \leq 
			\sum_{\substack{Q\in \Qcal_\delta}} \int_Q\Big| w(x) - \fint_Q w \Big| \, \dd x\\ 
			& =  \sum_{Q \in \mathcal Q_{\delta}}\delta^{n} \fint_{Q} \Big \vert w(x) - \fint_Q  w \Big\vert \, \dd x \\
			& = {\delta}  \Ksf_{\delta}(w).\qedhere
		\end{align*}
	\end{proof}
	
	The last necessary piece for the proof of Proposition~\ref{prop:compactness} is the following:
	\begin{lemma}\label{lemma:delta_vs_delta_mezzi}
		Let $w\in L^1_\loc(\R^n)$. Then
		\[
		\|w-w^{\delta/2}\|_{L^1(\R^n)}\leq (1+2^n) \|w-w^\delta\|_{L^1(\R^n)}.
		\]
	\end{lemma}
	
	\begin{proof}
		Let us first prove that for every $Q\in \Qcal_\delta$ it holds
		\begin{equation}\label{eq:local_difference_average}
			\|w-w^{\delta/2}\|_{L^1(Q)}\leq (1+2^n) \|w-w^{\delta}\|_{L^1(Q)}.
		\end{equation}
		Given a $\delta$-cube $Q\in \Qcal_\delta$, we denote by $\Ccal(Q)$ the family of $\tfrac{\delta}{2}$-subcubes of $Q$, namely those cubes in $\Qcal_{\delta/2}$ contained in $Q$.  
		Since we can express $w_{Q}$ as a convex sum as
		\[
		w_Q=\frac{1}{2^n}\sum_{Q'\in \Ccal(Q)}w_{Q'},
		\]
		we deduce that for every  $Q'\in \Ccal(Q)$ there exists $Q''\in\Ccal(Q)$ such that $|w_{Q'}-w_Q|\leq |w_{Q'}-w_{Q''}|$. Applying \eqref{eq:difference_averages} with $Q'$ and $Q''$ we obtain
		\[
		|w_Q-w_{Q'}|\leq |w_{Q'}-w_{Q''}|\leq \frac{1}{|Q'|}\int_Q |w(x)-w_Q|\, \dd x\qquad\text{for every $Q'\in \Ccal(Q)$.}
		\]
		Using the triangle inequality we get
		\begin{align*}
			\|w-w^{\delta/2}\|_{L^1(Q)}&\leq \|w-w^{\delta}\|_{L^1(Q)}+\|w^\delta-w^{\delta/2}\|_{L^1(Q)}\\
			&\leq \|w-w^{\delta}\|_{L^1(Q)}+\sum_{Q'\in \Ccal(Q)} |Q'|\,|w_Q-w_{Q'}|\\
			&\leq \|w-w^{\delta}\|_{L^1(Q)}+\sum_{Q'\in \Ccal(Q)} \int_{Q} |w(x)-w_Q|\, \dd x\\
			&\leq (1+2^n) \|w-w^{\delta}\|_{L^1(Q)}.
		\end{align*}
		Summing up over all cubes $Q \in \Qcal_\delta$ we obtain the sought global estimate.
	\end{proof}
	
	\subsection{Compactness.}
	
	We are finally ready to give the proof of the compactness property stated in Proposition~\ref{prop:compactness}. 
	
	\begin{proof}[Proof of Proposition~\ref{prop:compactness}.]
		It is not restrictive to assume that
		\[
		\lim_{j\to \infty} {\mathsf K}_{\e_j}(f_{\eps_j}) = \liminf_{j\to\infty} {\mathsf K}_{\e_j}(f_{\eps_j}) < \infty.
		\]
		For simplicity, we will denote from now onward $f_j \coloneqq f_{\e_j}$ for every $j \in \mathbb N$. Recalling the notation introduced in \eqref{eq:def_w_delta_in_every_dimension}, by Lemmas~\ref{lemma:e_close} and \ref{lemma:delta_vs_delta_mezzi} (with $c(n)=(1+2^n)$) we get
		\[
		\limsup_{j \to \infty} \| f_j - f_j^{\e_j/2} \|_{L^1(\R^n)} \le c(n)  \lim_{j \to \infty} \eps_j \Ksf_{\e_j}(f_{\e_j}) = 0.
		\]
		Hence, we conclude
		\begin{equation}\label{eq:intermediate} 
			f_j - f_j^{\e_j/2} \to 0 \quad \text{in $L^1(\R^n)$ as $j \to \infty$}. 
		\end{equation}  
		On the other hand, Lemma \ref{lemma:properties_of_piecewise_general_dimension} (iii) with $w = f_j$ and $\delta = \frac{\e_j}{2}$ gives 
		\begin{equation}\label{eq:bound_TV}
			\limsup_{j \to \infty} |D f_j^{\e_j/2}|(\R^n) \le 4n \lim_{j \to \infty} \Ksf_{\e_j}(f_j) < \infty.
		\end{equation}
		In particular, the sequence $(f_j^{\e_j/2})_j$ has uniformly bounded total variation on $\R^n$. Resorting to \cite[Thm.~3.47]{AFP} we conclude that there exists a sequence $(c_j)_j \subset \R$ such that for every $j \in \mathbb N$ it holds 
		\begin{equation}\label{eq:3.47}
			\| f_j^{\e_j/2} - c_j \|_{L^{1^*}(\R^n)} \le \gamma(n) |D f_j^{\e_j/2}|(\R^n)
		\end{equation}
		for some dimensional constant $\gamma(n)>0$. Here we have denoted by $1^*$ the Sobolev embedding exponent 
		\begin{equation*}
			1^* = \begin{cases}   
				\frac{n}{n-1} & \text{if $n\ge 2$} \\
				\,\,\infty & \text{if $n = 1$} 
			\end{cases}.
		\end{equation*}
		
		We apply the $BV$ compactness theorem~\cite[Thm.~3.23]{AFP} to conclude that there exists a function $f \in L^1_\loc(\R^n)$ and a subsequence (not relabeled) of the $j$'s such that
		\[
		f_j^{\e_j/2} - c_j \to f \quad\text{in $L^1_\loc(\R^n)$ as $j \to \infty$}.
		\]
		In light of~\eqref{eq:intermediate} we also conclude that 
		\[
		f_j - c_j \to f \quad \text{in $L^1_\loc(\R^n)$ as $j \to \infty$.} 
		\]
		This proves the first assertion.
		
		Let us show that $f\in L^{1^*}(\R^n)$. By \eqref{eq:3.47} the sequence $(f_j^{\e_j/2}-c_j)$ is equibounded in $L^{1^*}(\R^n)$, hence by Banach-Alaoglu's theorem it converges weakly star, up to a subsequence, to some $g\in L^{1^*}(\R^n)$. Since the sequence converges also in $L^1_\loc(\R^n)$ to $f$, it is straightforward to check that $f=g$ a.e., therefore $f\in L^{1^*}(\R^n)$.
		
		Finally, inequality \eqref{eq:tot_var_in_compactness} follows directly from \eqref{eq:bound_TV} and the lower semicontinuity of the total variation with respect to $L^1_\loc$-convergence.
	\end{proof} 
	
	\begin{remark}\label{rem:better_compactness} 
		Keeping in mind the close relationship between the $\mathsf K_\e$ functionals and the total variation, Proposition \ref{prop:compactness} can be thought of as an analog to Rellich's compactness theorem. However, while in the latter the convergence can be lifted to $L^p_\loc$ for every $1\leq p<1^*$, in our case the convergence \eqref{eq:convergence_compactness} cannot be improved to $L^{p}_\loc$ for any $p\neq 1$. Indeed, fix a function $\rho$ with support in the unit ball $B_1$ that belongs to $L^1\setminus L^p$ for every $p>1$. Then the family of functions $f_\e(x) \coloneqq  \rho(x/\e)$ converges to zero in $L^1$. Moreover, for every $\e$-cube $Q$ it holds
		\[
		\fint_{Q} \Big\vert f_\e(x) - \fint_{Q} f_\e\Big\vert \dd x\leq 2\fint_Q |f_\e|\leq 2\|\rho\|_{L^1}.
		\]
		Since the number of $\e$-cubes that intersect the support of $f_\e$ is bounded by a dimensional constant $M(n)$, then
		$\Ksf_\e(f_\e)\leq \e^{n-1}2M(n) \|\rho\|_{L^1(\R^n)}$. However, $f_\e\not\to 0$ in $L^p_\loc$ for any $p> 1$, because $\|f_\e\|_{L^p(B_1)}=\infty$ for every $\e>0$. A simple truncation of this example gives a counterexample even under a uniform $L^p$ bound on $(f_\e)_\e$.
	\end{remark}

	\section{Sharp lower bound in any dimension: a blow-up proof}\label{sec:lower_bound_general_case}
	
	Goal of this section is to prove the following sharp lower bound in any space dimension $n\ge 1$, improving the non-optimal constant $\tfrac{1}{4n}$ of Proposition~\ref{prop:compactness} to the optimal one $\tfrac14$.
	
	\begin{proposition}\label{prop:lower_bound_dimension_n}
		Let $(f_\eps)_{\eps>0} \subset L_\loc^1(\R^n)$ be such that $f_\e \to f$ in $L^1_\loc(\R)$ as $\eps \to 0$. Then 
		\begin{equation}\label{eq:lower_bound_dimension_n}
			\frac14 |Df|(\R^n)\leq \liminf_{\eps \to 0}\mathsf K_{\e}(f_\e).
		\end{equation}
	\end{proposition}
	
	\begin{proof} We will split the proof into several steps. 
		\vspace{2.5mm} 
		
		\emph{Step 1. Notations and preparations.}
		To shorten the notation, we will write 
		\[
		\Osc(w,Q) \coloneqq \fint_{Q} \Big\vert w(x) - \fint_{Q} w\Big\vert \dd x
		\]
		for every $w \in L^1_\loc(\R^n)$ and every cube $Q\subset \R^n$. 
		Furthermore, it is clearly sufficient to show \eqref{eq:lower_bound_dimension_n} along every infinitesimal sequence $\e_j \to 0$. Accordingly, we fix any such sequence and set $f_j \coloneqq f_{\e_j}$. Observe also that it is not restrictive to assume that $\liminf_{j \to \infty} \mathsf K_{\eps_j}(f_{j}) < \infty$ and, up to a non-relabeled subsequence, we will assume 
		\[
		\lim_{j \to \infty} \mathsf K_{\eps_j}(f_{j}) = \liminf_{j \to \infty} \mathsf K_{\eps_j}(f_{j}). 
		\]
		The thesis is thus equivalent to 
		\begin{equation}\label{eq:lower_bound_dimension_n_version_j}
			\frac14 |Df|(\R^n)\leq \lim_{j \to \infty}\mathsf K_{\e_j}(f_j).
		\end{equation}
		Finally, in view of Proposition~\ref{prop:compactness}, we may assume that $|Df|(\R^n)< \infty$.
		\vspace{2.5mm} 
		
		\emph{Step 2. Definition of the measures and compactness.} For each $\e>0$, let $\overline{\mathcal H}_{\eps}$ be an almost maximizing family of disjoint ${\eps}$-cubes, i.e., such that 
		\begin{equation}\label{eq:K_eps_minus_eps}
			\Ksf_{\eps}(f_{\eps}) - \eps \le \eps^{n-1} \sum_{Q \in \overline{\mathcal H}_{\eps}} \Osc(f_\eps,Q) \le \Ksf_{\eps}(f_{\eps}).
		\end{equation}
		We now define a family of non-negative measures, absolutely continuous with respect to $\Lcal^n$, which have piecewise constant density. We set for $\eps>0$ 
		\[
		\mu_{\eps} := \eps^{n-1}\sum_{Q \in \overline{\mathcal H}_{\eps}} \frac{\Osc(f_\eps,Q) }{\Lcal^n(Q)} \cdot \1_Q\, \Lcal^n.
		\]
		By construction, it holds 
		\[
		\mu_\eps(\R^n) = \eps^{n-1}\sum_{Q \in \overline{\mathcal H}_{\eps}} \frac{\Osc(f_\eps,Q) }{\Lcal^n(Q)} \Lcal^n(Q) \le \Ksf_{\eps}(f_{\eps}),
		\]
		hence $(\mu_{\e_j})_{j}$ is a sequence of equi-bounded non-negative measures and therefore also pre-compact with respect to the weak-$*$ convergence of measures. Upon extracting a subsequence, we may assume that
		\[
		\mu_{\eps_j} \overset{*}{\rightharpoonup} \mu
		\]
		weak-$*$ in the sense of measures for a non-negative measure $\mu$ on $\R^n$. 
		\vspace{2.5mm} 
		
		\emph{Step 3. A Radon--Nikodym inequality.} Since both $\mu$ and $|Df|$ are non-negative Radon measures, we can consider the decomposition 
		\begin{equation*}
			\mu=\frac{\dd\mu}{\dd|Df|}|Df|+\mu^*,
		\end{equation*} 
		where $\mu^*$ is the singular part of $\mu$ with respect to $|Df|$. The derivative density
		\begin{equation}\label{eq:RN_mu_Df}
			\frac{\dd\mu}{\dd |Df|}(x) := \lim_{r \to 0} \frac{\mu(B_r(x))}{|Df|(B_r(x))}\,,
		\end{equation}
		is well-defined $|Df|$-a.e. and defines a non-negative function in $L^1(\R^n,|Df|)$. The strategy of our proof resides in establishing the following lower bound for the Radon--Nikodym's derivative: 
		\begin{equation}\label{eq:mu_Du_geq_14}
			\frac{\dd \mu}{\dd |Df|} \ge \frac 14 \qquad \text{$|Df|$-almost everywhere}.
		\end{equation}
		Let us first show that, once \eqref{eq:mu_Du_geq_14} is proven, then \eqref{eq:lower_bound_dimension_n_version_j} easily follows. Indeed, since $\mu$ and $\mu^*$ are non-negative measures, we have the following chain of inequalities:
		\begin{align*}
			\lim_{j \to \infty} \mathsf K_{\eps_j}(f_j) & \ge \liminf_{j \to \infty} \mu_{\eps_j}(\R^n) \\
			& \geq \mu(\R^n)\\ 
			&=\int_{\R^n} \frac{\dd \mu}{\dd |Df|}(x)\, \dd |Df|(x) + \mu^*(\R^n)  \\
			& \ge \int_{\R^n} \frac{\dd \mu}{\dd |Df|}(x) \, \dd |Df|(x) \\
			& \ge \frac 14 |Df|(\R^n),
		\end{align*}
		where in the last inequality, we have used \eqref{eq:mu_Du_geq_14}. In the remaining steps of the proof, we will thus focus on proving \eqref{eq:mu_Du_geq_14}. We will achieve this goal by appealing to the mono-directionality property of $BV$ functions expressed in Lemma \ref{lemma:monodirectional}. 
		\vspace{2.5mm} 
		
		\emph{Step 4. Computation of Radon--Nikodym derivative.} Let us consider the function $f$ and its polar $\nu$, defined in \eqref{eq:def_polar}. Recall that, for every $x_0 \in \R^n$ such that the limit in \eqref{eq:def_polar} exists and has unit norm, we denote it by $\nu_0 := \nu(x_0)$.  
		For every $r>0$, we introduce the cube 
		\[
		Q_r^{\nu_0}(x_0) = Q_r(x_0) \coloneqq x_0 + rQ_1, 
		\]
		where $Q_1$ denotes any open cube of side-length $2$, centered at $0$ and with one of its $(d-1)$-dimensional faces orthogonal to $\nu_0$.
		
		Recalling Remark~\ref{rem:general_lebesgue}, up to a $|Df|$-negligible set $S \subset \R^n$, every point $x_0 \in \R^n$ satisfies the conclusion of Lemma \ref{lemma:monodirectional} with $B_r(x_0)$ replaced by $Q_r(x_0)$. Therefore, by \eqref{eq:RN_mu_Df} and Lemma \ref{lemma:monodirectional}, for $|Df|$-a.e. point $x_0$ it holds  
		\begin{equation}\label{eq:RN}
			\begin{split}
				\frac{\dd \mu}{\dd |Df|}(x_0) &
				= \frac{\dd \mu}{\dd |D_{\nu_0} f|}(x_0) \\
				&
				= \lim_{k \to \infty} \frac{\mu(Q_{r_k}(x_0))}{|D_{\nu_0} f|(Q_{r_k}(x_0))}\\
				& 
				= \lim_{k \to \infty} \, \lim_{j \to \infty} \frac{\mu_{\eps_j}(Q_{r_k}(x_0))}{|D_{\nu_0} f|(Q_{r_k}(x_0))},
			\end{split}
		\end{equation}
		provided that we choose an infinitesimal sequence $(r_k)_k$ so that $\mu(\partial Q_{r_k}(x_0))=0$ (see \cite[Prop.~1.62(b)]{AFP}). 
		
		Let us fix from now on  such an $x_0$. Up to a rotation, we may assume $\nu_0=e_1=(1,0, \ldots, 0)$ and thus $Q_1 = \left(-\frac12,\frac12\right)^{n}$. Recall that the family $\mathcal Q_\e$ was defined in \eqref{eq:def_mesh_cubes} as 
		\[
		\Q_\e:=\{(0,\e)^n+\e z : z \in \Zbb^n\} 
		\]
		and its translated by $\tau\in \R^n$ was defined in \eqref{eq:def_mesh_cubes_translated} as 
		\[
		\tau+\Q_\e\coloneqq \{\tau+Q:Q\in\Q_\e\}.
		\]
		Let us now set 
		\[
		2\delta \coloneqq \eps.
		\]
		For a given $P\in\Q_\e=\Q_{2\delta}$, let us denote by $\Ncal_\delta(P,\nu_0)$ the family of unordered pairs $\{Q,Q'\}$ of distinct cubes $Q,Q'\in \Q_\delta$ that are contained in $P$ and that share an $(n-1)$-dimensional face parallel to $\nu_0^\perp$.
		
		We now set for simplicity $w\coloneqq f_\e$ and write
		\[
		w^{\delta} \coloneqq \sum_{Q \in \Qcal_\delta} \1_Q w_Q,\quad w_Q\coloneqq \fint_Q w.
		\]
		Applying \eqref{eq:jump_betwen_Q_Q'} with $m=w_P$ we have 
		\begin{align}\label{eq:estimate_on_P}
			\begin{aligned}
				|D_{\nu_0} w^\delta|(P)&= \sum_{\{Q,Q'\}\in \Ncal_\delta(P,\nu_0)} |w_Q-w_{Q'}|\delta^{n-1}\\
				&\leq \delta^{n-1}\sum_{\{Q,Q'\}\in \Ncal_\delta(P,\nu_0)} \left(\fint_Q|w-w_P|+\fint_{Q'}|w-w_P|\right)\\
				&= \delta^{n-1}\sum_{\{Q,Q'\}\in \Ncal_\delta(P,\nu_0)} \frac{1}{\delta^n}\int_{Q\cup Q'}|w-w_P|\\
				&=2^n\delta^{n-1}\fint_P|w-w_P|
			\end{aligned}
		\end{align}
		where we have used that $\Lcal^n(P)=2^n\Lcal^n(Q)=2^n\Lcal^n(Q')$ and that, up to $\Lcal^n$-negligible sets, $P$ is the disjoint union
		\[
		P=\bigcup_{\{Q,Q'\}\in \Ncal_\delta(P,\nu_0)}Q\cup Q'.
		\]
		
		\begin{figure}
			\captionsetup[subfigure]{labelformat=empty,width=.6\columnwidth}
			\subfloat[The original mesh $\Qcal_{\e}$ (dashed) and the translated mesh $\delta \nu_0 + \Qcal_{\e}$ (gray). The directional derivative of $w^\delta$ in direction $\nu_0$ sees only jumps across 1-dimensional faces (of cubes in $\Q_\delta$) that are parallel to $\nu_0^\perp$. Such faces appear exactly once in the interior of a cube $P$ belonging to either $\Q_{\e}$ or $\delta\nu_0+\Q_{\e}$.]{
				\begin{tikzpicture}[scale=.8]
					\draw[thick,white,->] (-5,0) -- (-4,0) node[anchor=south]{\tiny{$\nu_0$}};
					\filldraw[white] (0,-4) circle (2.5pt);
					\draw[thick,->] (4,0) -- (5,0) node[anchor=south]{\tiny{$\nu_0$}};
					\draw[thick,->] (4,0) -- (4,1) node[anchor=south]{\tiny{$\nu_0^\perp$}};
				
					\foreach \x in {-3,-1,...,3}{
						\draw[black!10, line width=2pt] (\x,-3.5) -- (\x,3.5);
					}
					\foreach \y in {-2,0,2}{
						\draw[black!10, line width=2pt] (-3.5,\y) -- (3.5,\y);
					}
					
					\foreach \x in {-2,0,2}{
						\draw[black, loosely dashed, thick] (\x,-3.5) -- (\x,3.5);
					}
					\foreach \y in {-2,0,2}{
						\draw[black, loosely dashed, thick] (-3.5,\y) -- (3.5,\y);
					}
					
					
					\draw[->,  very thick] (0,1) -- (1,1);
					\filldraw (0.5,1.5) circle (.0pt) node {\scriptsize  $\delta$};
					
					\filldraw (0,0) circle (2.5pt);
				\end{tikzpicture}
			}
			\caption{}\label{fig:horizontal_translation}
		\end{figure}
		
		Observe that every $(n-1)$-dimensional face (of a cube in $\Q_\delta$) that is orthogonal to $\nu_0$ appears exactly once in the interior of a cube $P$ belonging to either $\Q_{\e}$ or $\delta\nu_0+\Q_{\e}$, see Fig. \ref{fig:horizontal_translation}. Thus by \eqref{eq:estimate_on_P} we have
		\begin{align}
			|D_{\nu_0} w^\delta|(&Q_{r-\sqrt{n}\e}(x_0))= \sum_{\substack{P\in \Q_\e\\ P\subseteq Q_{r-\sqrt{n}\e}(x_0)}}|D_{\nu_0}w^\delta|(P)+\sum_{\substack{P\in \delta\nu_0+\Q_\e\\ P\subseteq Q_{r-\sqrt{n}\e}(x_0)}}|D_{\nu_0}w^\delta|(P)\nonumber\\
			&\leq \sum_{\substack{P\in \Q_\e\\ P\subseteq Q_{r-\sqrt{n}\e}(x_0)}} 2^n\delta^{n-1}\Osc(w,P)+\sum_{\substack{P\in \delta\nu_0+\Q_\e\\ P\subseteq Q_{r-\sqrt{n}\e}(x_0)}}2^n\delta^{n-1}\Osc(w,P).\label{eq:split_odd_even}
		\end{align}
		Let us estimate the first term from above; the estimate for the second is analogous. First, let us split the family $\overline{\Hcal}_\e$ as $\overline{\Hcal}_\e=\overline{\Hcal}_\e^{in}\cup\overline{\Hcal}_\e^{out}$, where
		\[
		\overline{\Hcal}_\e^{in}\coloneqq \{P\in \overline{\Hcal}_\e: P\cap Q_{r-\sqrt{n}\e}(x_0)\neq \emptyset\},
		\]
		\[\overline{\Hcal}_\e^{out}\coloneqq \{P\in \overline{\Hcal}_\e: P\cap Q_{r-\sqrt{n}\e}(x_0)= \emptyset\}.
		\]
		We now construct the following modification of the family $\overline{\Hcal}_\e$ inside $Q_r$:
		\[
		\overline{\Ocal}_\e^{in}\coloneqq \{P\in \Q_\e : P\subseteq Q_{r-\sqrt{n}\e}(x_0)\},\qquad \overline{\Ocal}_\e^{out}\coloneqq \overline{\Hcal}_\e^{out}.
		\]
		Since $\overline{\Ocal}_\e\coloneqq \overline{\Ocal}_\e^{in}\cup\overline{\Ocal}_\e^{out} $ is a family of disjoint $\e$-cubes, recalling \eqref{eq:K_eps_minus_eps} it certainly holds
		\[
		\e^{n-1}\sum_{P\in \overline{\Ocal}_\e} \Osc(w,P)\leq \mathsf K_\e(w)\leq \e+\e^{n-1}\sum_{P\in \overline{\Hcal}_\e}\Osc(w,P),
		\]
		whence, splitting the sums and simplifying the sum over $\overline{\Ocal}_\e^{out}= \overline{\Hcal}_\e^{out}$, we obtain
		\begin{equation}\label{eq:burrata}
			\e^{n-1}\sum_{P\in \overline{\Ocal}_\e^{in}} \Osc(w,P)\leq \e+\e^{n-1}\sum_{P\in \overline{\Hcal}_\e^{in}}\Osc(w,P).
		\end{equation}
		We observe now that by construction every cube $P\in\overline{\Hcal}_\e^{in}$ is contained in $Q_r$  so that
		\begin{equation}\label{eq:estimate_with_mu_eps}
			\e^{n-1}\sum_{P\in \overline{\Hcal}_\e^{in}}\Osc(w,P)\leq \mu_\e (Q_r).
		\end{equation}
		Combining \eqref{eq:burrata} and \eqref{eq:estimate_with_mu_eps} we conclude that
		\[
		\sum_{\substack{P\in \Q_\e\\ P\subseteq Q_{r-\sqrt{n}\e}(x_0)}} 2^n\delta^{n-1}\Osc(w,P)=2 \e^{n-1}\sum_{P\in \overline{\Ocal}_\e^{in}} \Osc(w,P)\leq 2(\e +\mu_\e(Q_r)).
		\]
		Since an analogous estimate holds for the second sum in \eqref{eq:split_odd_even} we conclude that
		\begin{equation}\label{eq:estimate_total_derivative}
			|D_{\nu_0} w^\delta|(Q_{r-\sqrt{n}\e}(x_0))\leq 4(\e+\mu_\e(Q_r)).
		\end{equation}
		
		\vspace{2.5mm} 
		\emph{Step 5. Conclusion.} 
		Recalling that $w=f_\e$ and $2\delta=\e$, we have that $w^\delta \to f$ in $L^1_\loc$ as $\delta \to 0$. Indeed, for any cube $D_R = (-R,R)^n$ and $\delta$ sufficiently small we get by Lemma \ref{lemma:properties_of_piecewise_general_dimension} (i)
		\begin{align*}
			\Vert w^\delta - f \Vert_{L^1(D_R)} & \le \Vert f^{\e/2}_\e - f^{\e/2} \Vert_{L^1(D_R)} + \Vert  f^{\e/2} -f \Vert_{L^1(D_R)} \\
			& \le \Vert f_\e - f \Vert_{L^1(D_{R+1})} + \Vert  f^{\e/2} -f \Vert_{L^1(D_R)}.
		\end{align*}
		As $\delta \to 0$ (so that $\e \to 0$), the first term in RHS goes to zero by assumption, 
		while the second term tends to zero thanks to Lemma \ref{lemma:properties_of_piecewise_general_dimension} (ii). This proves that $\omega^\delta \to f$ in $L^1_\loc$.
		
		By the lower semicontinuity of the total variation and \eqref{eq:estimate_total_derivative}, for any fixed $\Omega\Subset Q_r(x_0)$ we therefore have the inequality
		\begin{align*}
			|D_{\nu_0} f|(\Omega) & \le  \liminf_{\delta \to 0} |D_{\nu_0} w^\delta|(\Omega) \\
			& \le \liminf_{\delta_j \to 0} |D_{\nu_0} w^{\delta_j}|(Q_{r-\sqrt{n}\e_j}(x_0)) \\
			&\leq \liminf_{\delta_j\to 0} 4(\e_j+\mu_{\e_j}(Q_r(x_0)))\\
			&= 4 \mu(Q_r(x_0)).
		\end{align*}
		Taking the supremum over all compact subsets of $Q_{r}(x_0)$ we conclude that
		\[
		|D_{\nu_0} f|(Q_{r}(x_0)) \le  4 \mu(Q_{r}(x_0)),
		\]
		which in light of~\eqref{eq:RN} gives the estimate
		\[
		\frac{\dd \mu}{\dd |Df|}(x_0) \ge \frac 14 \qquad \text{for all $x_0 \in S$}.
		\]
		This proves~\eqref{eq:mu_Du_geq_14} and the proof is complete.
	\end{proof}

	\section{\texorpdfstring{$\Gamma$}{Gamma}-limsup inequality}\label{sec:upper_bound}
	To complete the proof of Theorem~\ref{thm:main} we need to establish the $\Gamma$-limsup inequality. This task is addressed in this section. The following lemma, which follows directly from \cite[Thm.~3.3]{FMS}, illustrates the (pointwise) behavior of the energies $\mathsf{K}_\e$ on smooth functions.
	
	\begin{lemma}\label{lemma:smooth} 
		If $f \in C^\infty(\R^n)$, then 
		\[
		\lim_{\e\to 0}\mathsf{K}_\e(f)=\frac14 \int_{\R^n} |\nabla f(x)| \, \dd x.
		\]
	\end{lemma}
	
	We can now prove the $\Gamma$-limsup inequality. 
	
	\begin{proposition}[$\Gamma$-limsup]\label{prop:gamma-limsup} For every $f \in L^1_\loc(\R^n)$, there exists a family $(f_\e)_{\e}$ such that $f_\e \to f$ in $L^1_\loc(\mathbb R^n)$ and 
		\begin{equation}\label{eq:gamma-limsup}
			\limsup_{\e \to 0} \mathsf K_{\e} (f_\e) \le \frac{1}{4} |Df|(\R^n). 
		\end{equation}
	\end{proposition}
	
	\begin{proof}
		Observe that it is not restrictive to assume that $|Df|(\R^n)<\infty$, otherwise there would be nothing to prove.  By standard facts of $\Gamma$-convergence, it is enough to prove \eqref{eq:gamma-limsup} on a subset $\Dcal \subset \{g \in C^\infty(\R^n) : |Dg| <\infty\}$ that is dense with respect to the topology induced by the $L^1_\loc$ metric and the convergence of the total variation seminorm. 
		By standard mollification arguments we can select $\Dcal =C^\infty(\mathbb R^n)$. Then, for an arbitrary $f \in \Dcal$ we simply consider the family consisting of the elements $f_{\e}:=f$ for every $\e>0$. By Lemma~\ref{lemma:smooth} we conclude that
		\begin{equation*}
			\limsup_{\e\to 0} \mathsf K_{\e}(f_{\e}) = \limsup_{\e\to 0} \mathsf K_{\e}(f) = \frac{1}{4}|Df|(\R^n).
		\end{equation*}
		and this yields the desired conclusion. 	
	\end{proof}

\end{document}